\documentclass[11pt,fleqn]{article}

\usepackage[english]{babel}
\usepackage{verbatim}
\usepackage{amsmath}
\usepackage{amsthm}
\usepackage{amsfonts}
\usepackage{amssymb}

\usepackage[utf8]{inputenc}
\usepackage{mathrsfs}
\usepackage{paralist}

\usepackage{subfigure}

\usepackage{pgf}
\usepackage{tikz} 
\usetikzlibrary{arrows,shapes,positioning}
\usetikzlibrary{decorations.markings}
\usetikzlibrary{arrows,shapes,positioning}
\usetikzlibrary{decorations.markings}

\theoremstyle{plain}

\newtheorem{theorem}{Theorem}[section]

\newtheorem{lemma}[theorem]{Lemma}
\newtheorem{corollary}[theorem]{Corollary}

\theoremstyle{definition}

\newtheorem{example}[theorem]{Example}

\newcommand{\C}{\mathbb{C}}
\newcommand{\R}{\mathbb{R}}

\newcommand{\cK}{\mathcal{K}}
\newcommand{\cL}{\mathcal{L}}

\newcommand{\cO}{\mathcal{O}}

\newcommand{\abs}[1]{\vert #1 \vert}

\DeclareMathOperator{\cn}{cn}

\DeclareMathOperator{\dn}{dn}

\DeclareMathOperator{\sn}{sn}

\DeclareMathOperator{\exterior}{ext}

\newcommand{\coloneq}{\mathrel{\mathop:}=}
\newcommand{\eqcolon}{=\mathrel{\mathop:}}

\newcommand{\eop}{}

\allowdisplaybreaks[1]

\usepackage{hyperref}

\author{Olivier S\`{e}te\footnotemark[1] \and J\"{o}rg Liesen\footnotemark[1]}
\title{On conformal maps from multiply connected domains onto lemniscatic 
domains}

\begin{document}
\maketitle

\renewcommand{\thefootnote}{\fnsymbol{footnote}}

\footnotetext[1]{Technische Universit\"{a}t Berlin, Institute of Mathematics, 
MA~4-5, Stra{\ss}e des 
17. Juni 136, 10623 Berlin, Germany.
\texttt{\{sete,liesen\}@math.tu-berlin.de}}

\renewcommand{\thefootnote}{\arabic{footnote}}

\begin{abstract}
We study conformal maps from multiply connected domains in the extended complex 
plane onto lemniscatic domains.  Walsh proved the existence of such maps in 
1956 and thus obtained a direct generalization of the Riemann mapping theorem 
to multiply connected domains.
For certain polynomial pre-images of simply connected sets we derive a 
construction principle for Walsh's conformal map in terms of the Riemann map 
for the simply connected set.
Moreover, we explicitly construct examples of Walsh's conformal map for certain 
radial slit domains and circular domains.

\end{abstract}




\section{Introduction}

Let $\cK$ be any simply connected domain (open and connected set) in the 
extended complex plane $\widehat{\C}$ with $\infty \in \cK$ and with at least 
two boundary points.  Then the Riemann mapping theorem guarantees the existence 
of a conformal map $\Phi$ from $\cK$ onto the exterior of the unit disk, which 
is uniquely determined by the normalization conditions $\Phi(\infty)=\infty$ 
and $\Phi'(\infty) > 0$.
The exterior of the unit disk therefore is considered \emph{the canonical
domain} every such domain $\cK$ can be conformally identified with (in the
Riemann sense).  For domains $\cK$ that are not simply connected the conformal 
identification with a suitable canonical domain is significantly more challenging.
This fact has been well described already by Nehari in his classical monograph on
conformal mappings from 1952~\cite[Chapter~7]{Nehari1952}, which identified
five of the ``more important'' canonical slit domains (originally due to
Koebe~\cite[p.~311]{Koebe1916}). 

In recent years there has been a surge of
interest in the theory and computation of conformal maps for multiply connected
sets, which has been driven by the wealth of applications of conformal mapping
techniques throughout the mathematical sciences. Many recent publications have
dealt with canonical slit domains as those described by Nehari; see,
e.g.,~\cite{AndreevMcnicholl2012,CrowdyMarshall2006,DelilloDriscollElcratPfaltzgraff2008,DelilloKropf2011,Nasser2011,Nasser2013}.
A related line of recent research in this context has focussed on the theory
and computation of Schwarz-Christoffel mapping formulas from (the exterior of)
finitely many non-intersecting disks (circular domains, see,
e.g.,~\cite{Henrici1986}) onto (the exterior of) the same number
of non-intersecting polygons; see,
e.g.,~\cite{Crowdy2005,Crowdy2007,Delillo2006,DelilloDriscollElcratPfaltzgraff2006,DelilloElcratPfaltzgraff2004}.
A review and comparison of both approaches is given in~\cite{DelilloKropf2009}.

In this work we explore yet another idea which goes back to a paper
of Walsh from 1956~\cite{Walsh1956}. Walsh's canonical domain is a
\emph{lemniscatic domain} of the form
\begin{equation}\label{eqn:lemniscate}
\cL \coloneq \{ w \in \widehat{\C} : \abs{U(w)} > \mu \}, \quad \text{where} 
\quad U(w) \coloneq \prod_{j=1}^n (w-a_j)^{m_j},
\end{equation}
$a_1, \ldots, a_n\in\C$ are pairwise distinct,
$m_1, \ldots, m_n > 0$ satisfy $\sum_{j=1}^n m_j = 1$, and $\mu > 0$.
Note that the function $U$ in the definition of $\cL$ is an analytic but in
general multiple-valued function. Its absolute value is, however,
single-valued.
Walsh proved that if $\cK$ is the exterior of $n \geq 1$ non-intersecting 
simply connected components, then $\cK$ can be conformally identified
with \emph{some} lemniscatic domain $\cL$ of the form \eqref{eqn:lemniscate};
see Theorem~\ref{thm:existence_ext_mapping} below for the complete statement.
Walsh's theorem is a direct generalization of the Riemann mapping theorem,
and for $n=1$ the two results are in fact equivalent. Alternative
proofs of Walsh's theorem were given by Grunsky~\cite{Grunsky1957,Grunsky1957a}
(see also~\cite[Theorem~3.8.3]{Grunsky1978}), Jenkins~\cite{Jenkins1958} and
Landau~\cite{Landau1961}. For some further remarks on Walsh's theorem
we refer to Gaier's commentary in Walsh's Selected
Papers~\cite[pp.~374-377]{Walsh2000}.

To our knowledge, apart from the different \emph{existence proofs},
conformal maps related to Walsh's lemniscatic domains, which we call 
\emph{lemniscatic maps}, have rarely been studied. In particular, we 
are not aware of any example for lemniscatic maps in the previously 
published literature. In this work we derive a general construction
principle for lemniscatic maps for polynomial pre-images of simply
connected sets and we construct some explicit examples.
We believe that our results are of interest not only from a theoretical 
but also from a practical point of view. Walsh's lemniscatic 
map easily reveals the logarithmic capacity of $E = \widehat{\C} \backslash 
\cK$ 
as well as the Green's function with pole at infinity for $\cK$, whose contour 
lines or level curves are important in polynomial approximation.
Moreover, analogously to the construction of the 
classical Faber polynomials on compact and simply connected sets
(cf.~\cite{Curtiss1971,Suetin1998}), lemniscatic maps allow to define
generalized Faber polynomials on compact sets with several components;
see~\cite{Walsh1958}. While the classical Faber polynomials
have found a wide range of applications in particular in 
numerical linear
algebra (see, e.g.,~\cite{BeckermannReichel2009,HeuvelineSadkane1997,MoretNovati2001,MoretNovati2001a,StarkeVarga1993}) and
more general numerical polynomial approximation (see,
e.g.,~\cite{Ellacott1983,Ellacott1986}), the \emph{Faber--Walsh polynomials}
have not been used for similar purposes yet,
as no explicit examples for lemniscatic maps have been known. 
In our follow-up paper~\cite{SetLie15} we present more details on the theory 
of Faber--Walsh polynomials as well as explicitly computed examples.


In Section~\ref{sect:properties} we state Walsh's theorem, and discuss general 
properties of the conformal map onto lemniscatic domains.
We then consider the explicit construction of lemniscatic maps:
In Section~\ref{sect:pre-images} we derive a construction principle for the 
lemniscatic map for certain polynomial pre-images of simply connected compact 
sets. In Section~\ref{sect:map_disks} we construct the lemniscatic map for the 
exterior of two equal disks. Some brief concluding remarks in Section~\ref{sect:concl} 
close the paper.


\section{General properties of the conformal map onto lemniscatic domains}
\label{sect:properties}

Let us first consider a lemniscatic domain $\cL$ as in~\eqref{eqn:lemniscate}. 
It is easy to see that its Green's function with pole at infinity is given by
\begin{equation*} 
g_\cL(w) = \log \abs{U(w)} - \log(\mu).
\end{equation*}
Moreover,
\begin{equation*} 
c(\widehat{\C} \backslash \cL) \coloneq \lim_{w\rightarrow\infty}\, \exp(\log 
\abs{w} - g_\cL(w)) = \mu
\end{equation*}
is the logarithmic capacity of $\widehat{\C} \backslash \cL$. The following 
theorem on the conformal equivalence of lemniscatic domains and certain 
multiply connected domains is due to Walsh~\cite[Theorems~3 and~4]{Walsh1956}.

\begin{theorem} \label{thm:existence_ext_mapping}
Let $E \coloneq \cup_{j=1}^n E_j$, where $E_1, \ldots, E_n \subseteq \C$ are
mutually exterior simply connected compact sets (none a single point) and
let $\cK \coloneq \widehat{\C} \backslash E$.
Then there exist a unique lemniscatic domain $\cL$ of the 
form~\eqref{eqn:lemniscate} and a unique bijective conformal map
\begin{equation}\label{eqn:normalization_Phi}
\Phi : \cK \to \cL\quad\mbox{with}\quad
\Phi(z) = z + \cO \left( \frac{1}{z} \right) \quad \text{for $z$ near infinity}.
\end{equation}
In particular,
\begin{equation}\label{eqn:green_cK}
g_\cK(z) = g_\cL(\Phi(z)) = \log \abs{U(\Phi(z))} - \log(\mu)
\end{equation}
is the Green's function with pole at infinity of $\cK$, and the logarithmic 
capacity of $E$ is
$c(E)=c(\widehat{\C} \backslash \cL)=\mu$.
The function $\Phi$ is called the \emph{lemniscatic map} of $\cK$ (or of $E$).
\end{theorem}

Note that for $n=1$ the lemniscatic domain $\cL$ is the exterior of
a disk with radius $\mu>0$, and Theorem~\ref{thm:existence_ext_mapping} 
is equivalent to the classical Riemann mapping theorem.

If, for the given set $\cK$, the function 
$\widetilde{\Phi} : \cK \to \widetilde{\cL}$ is any conformal map onto 
a lemniscatic domain that is normalized by $\widetilde{\Phi}(\infty) = \infty$ 
and $\widetilde{\Phi}'(\infty) = 1$, then $\widetilde{\Phi}(z) = \Phi(z) + b$ 
with $\Phi$ from Theorem~\ref{thm:existence_ext_mapping} and some $b \in \C$.
This uniqueness up to translation of lemniscatic domains follows from a more 
general theorem of Walsh~\cite[Theorem~4]{Walsh1956} by taking into 
account the normalization of $\widetilde{\Phi}$. This fact has already been 
noted by Motzkin in his MathSciNet review of~\cite{Walsh1958}.

Let $\sigma > 1$, and let 
\begin{equation*}
\Gamma_\sigma = \{ z \in \cK : g_\cK(z) = \log(\sigma) \}\quad\mbox{and}
\quad
\Lambda_\sigma = \{ w \in \cL : g_\cL(w) = \log(\sigma) \}
\end{equation*}
be the level curves of $g_\cK$ and $g_\cL$, respectively.
Then~\eqref{eqn:green_cK} implies $\Phi(\Gamma_\sigma) = \Lambda_\sigma$, 
and thus
\begin{equation*}
\Phi : \exterior ( \Gamma_\sigma ) \to \exterior ( \Lambda_\sigma ) = \{ w \in 
\widehat{\C} : \abs{U(w)} > \sigma \mu \}
\end{equation*}
is the lemniscatic map of the exterior of $\Gamma_\sigma$, provided that 
$\Gamma_\sigma$ still has $n$ components.  (This holds exactly when the 
zeros of $g_\cK'$ lie exterior to $\Gamma_\sigma$.) 
Thus, we may ``thicken'' the given set $E = \widehat{\C} \backslash \cK$, 
and $\Phi$ still is the corresponding lemniscatic map.
An illustration is given in Figure~\ref{fig:thick_E} for a compact set
composed of three radial slits from Corollary~\ref{cor:radial_slits} below
(with parameters $n=3$, $C = 1$ and $D = 2$) and for $\Gamma_\sigma$ with
$\sigma = 1.15$.

\begin{figure}[t]
\begin{center}
\includegraphics[width=0.48\textwidth]{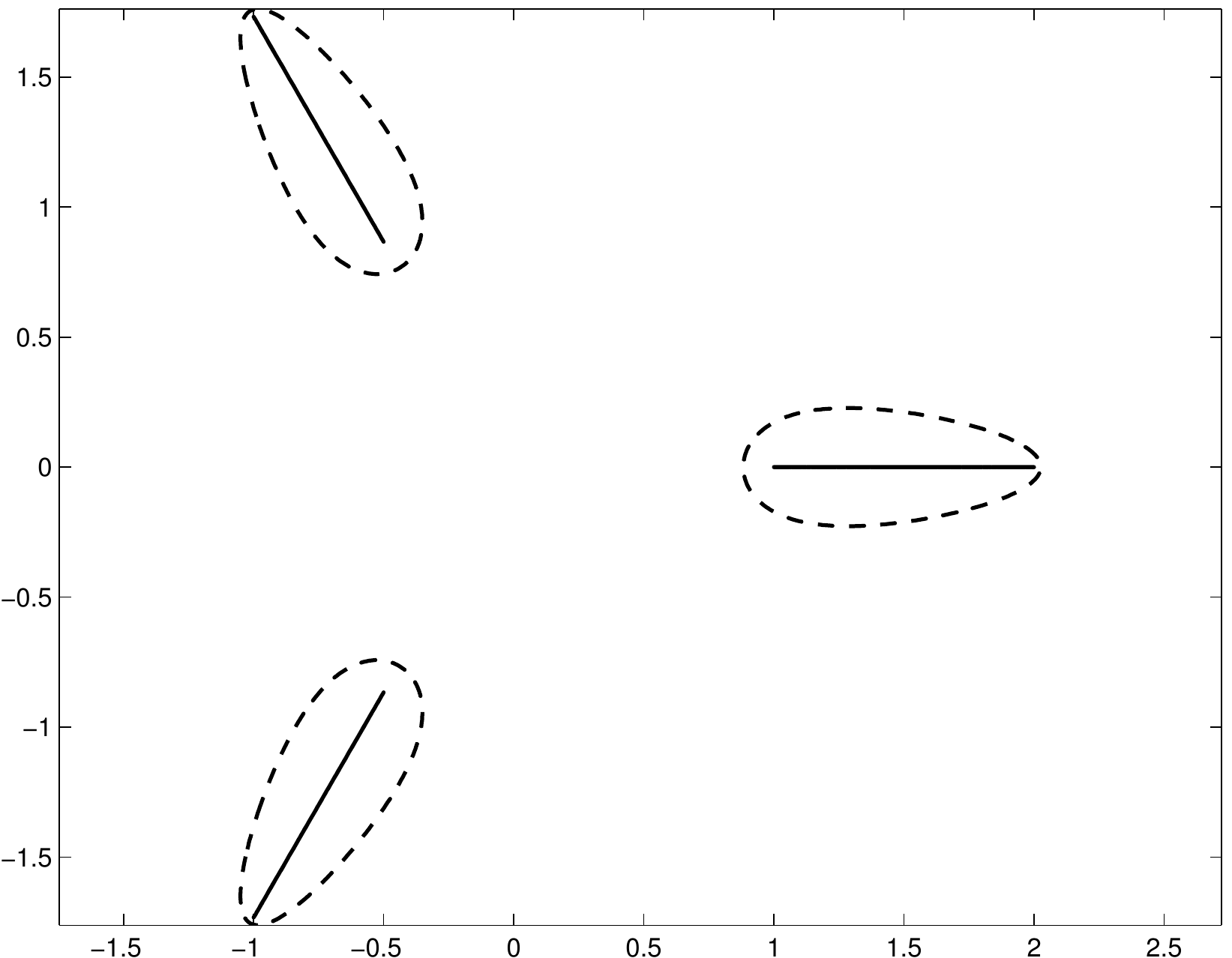}
\includegraphics[width=0.48\textwidth]{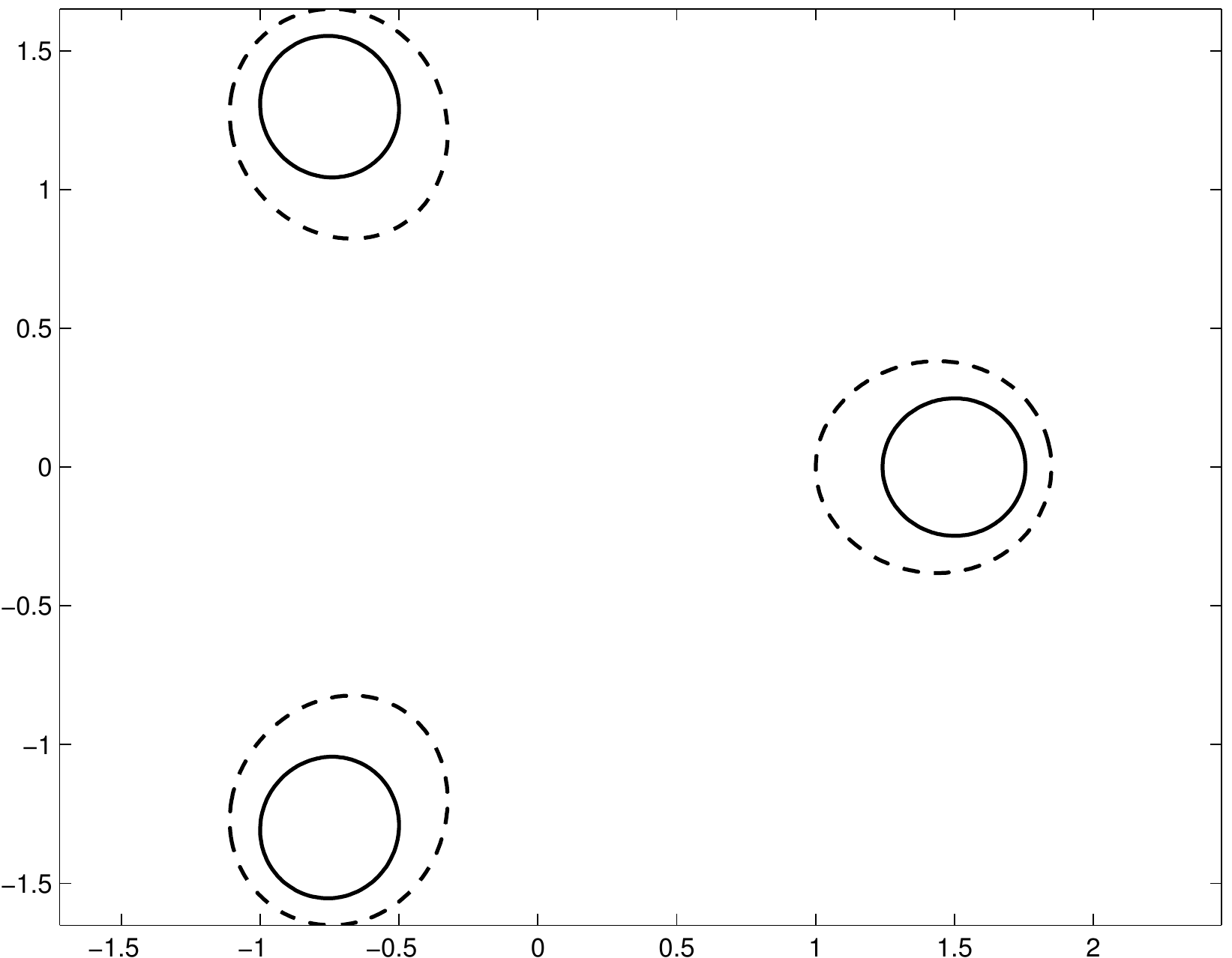}
\end{center}
\caption{Left: The set $E$ consisting of three radial slits (solid) and the 
``thickened'' set bounded by $\Gamma_\sigma$ for $\sigma=1.15$ (dashed).
Right: The corresponding lemniscatic domains.}
\label{fig:thick_E}
\end{figure}

The next result shows that certain symmetry properties of the domain $\cK$
imply corresponding properties of its lemniscatic map $\Phi$ and the 
lemniscatic domain $\cL$.  
Here we consider rotational symmetry as well as symmetry with respect to the 
real and the imaginary axis.

\begin{lemma} \label{lem:E_symm}
In the notation of Theorem~\ref{thm:existence_ext_mapping} we have:
\begin{enumerate}
\item If $\cK = e^{i \theta} \cK \coloneq \{ e^{i \theta} z : z \in \cK \}$, 
then $\Phi(z) = e^{- i \theta} \Phi( e^{i \theta} z )$ and $\cL = e^{i \theta} 
\cL$.

\item If $\cK = \cK^\ast \coloneq \{ \overline{z} : z \in \cK \}$, then 
$\Phi(z)= \overline{ \Phi(\overline{z}) }$ and $\cL = \cL^*$.

\item If $\cK = - \cK^*$, then $\Phi(z) = - \overline{ \Phi(-\overline{z}) }$ 
and $\cL = - \cL^*$.
\end{enumerate}
In each case $\Phi^{-1}$ has the same symmetry property as $\Phi$.
\end{lemma}

\begin{proof}
We only prove the first assertion; the proofs of the others are similar. 
Define the function $\widetilde{\Phi}$ on $\cK$
by $\widetilde{\Phi}(z) \coloneq e^{-i \theta} \Phi( e^{i \theta} z)$.  Then
\begin{equation*}
\widetilde{\Phi}(\cK) = e^{-i \theta} \Phi(e^{i \theta} \cK)
= e^{-i \theta} \Phi(\cK) = e^{-i \theta} \cL,
\end{equation*}
and $\widetilde{\Phi} : \cK \to e^{-i \theta} \cL$ is a bijective conformal map 
onto a 
lemniscatic domain with a normalization as in~\eqref{eqn:normalization_Phi}. 
Since the lemniscatic map of $\cK$ is unique, we have 
$\Phi(z) = \widetilde{\Phi}(z) = e^{-i \theta} \Phi(e^{i \theta} z)$ and $\cL = 
e^{-i \theta} \cL$, or equivalently $\cL = e^{i \theta} \cL$.

Suppose that $\Phi(z)= e^{-i \theta} \Phi(e^{i \theta} z)$ for all $z \in \cK$. 
Writing $w=\Phi(z)$ we get
\begin{equation*}
\Phi^{-1}(e^{i \theta} w) =\Phi^{-1}( e^{i \theta} \Phi(z))
=\Phi^{-1}(\Phi(e^{i \theta} z))= e^{i \theta} z= e^{i \theta} \Phi^{-1}(w),
\end{equation*}
which completes the proof.
\eop
\end{proof}

Finally, we show how a linear transformation of the set affects the 
lemniscatic map.

\begin{lemma} \label{lem:linear_trafo}
In the notation of Theorem~\ref{thm:existence_ext_mapping}, consider a
linear transformation $\tau(w) = a w + b$ with $a \neq 0$,  then
\begin{equation*}
\tau(\cL) = \Big\{ \widetilde{w} \in \widehat{\C} : \prod_{j=1}^n 
\abs{\widetilde{w} - \tau(a_j)}^{m_j} > \abs{a} \mu \Big\}
\end{equation*}
is a lemniscatic domain and $\widetilde{\Phi} \coloneq \tau \circ \Phi \circ
\tau^{-1}$ is the lemniscatic map of $\tau(\cK)$.
\end{lemma}

\begin{proof}
With $\widetilde{w} = \tau(w) = a w + b$ we have
\begin{equation*}
\prod_{j=1}^n \abs{\widetilde{w} - \tau(a_j)}^{m_j}
= \prod_{j=1}^n \abs{a w - a a_j}^{m_j}
= \abs{a} \prod_{j=1}^n \abs{w-a_j}^{m_j},
\end{equation*}
and hence $\tau(\cL)$ is a lemniscatic domain.  Clearly, $\widetilde{\Phi} :
\tau(\cK) \to \tau(\cL)$ is a bijective and conformal map with Laurent series 
at infinity
\begin{equation*}
\widetilde{\Phi}(z) = a \Phi \Big(\frac{z-b}{a}\Big) + b
= z + \cO \left(\frac{1}{z}\right).
\end{equation*}
Thus, $\widetilde{\Phi}$  is the lemniscatic map of $\tau(\cK)$.
\eop
\end{proof}

Lemma~\ref{lem:linear_trafo} can be applied to the lemniscatic maps 
we will derive in Sections~\ref{sect:pre-images} and~\ref{sect:map_disks} in 
order to obtain lemniscatic maps for further sets.


\section{Lemniscatic maps and polynomial pre-images}
\label{sect:pre-images}

In this section we discuss the construction of lemniscatic maps if the set $E$ 
is a polynomial pre-image of a simply connected compact set $\Omega$. We first 
exhibit the intricate relation between the lemniscatic map for $E$ and the 
exterior Riemann map for $\Omega$ in the general case. Under some additional 
assumptions we obtain an explicit formula for the lemniscatic map in terms of 
the Riemann map; see Theorem~\ref{thm:pre-images} below.

Let $\Omega\subseteq\C$ be a compact and simply connected set (not a single 
point) and let
\begin{equation}\label{eqn:Riemann_map}
\widetilde{\Phi}:\widehat{\C}\backslash\Omega\rightarrow \{w\in\widehat{\C}: 
|w|>1\} \quad
\text{with} \quad \widetilde{\Phi}(\infty) = \infty, \quad 
\widetilde{\Phi}'(\infty) > 0,
\end{equation}
be the exterior Riemann map of $\Omega$. Suppose that
\begin{equation*}
E \coloneq P^{-1}(\Omega)
\end{equation*}
consists of $n \geq 2$ simply connected compact components (none a single 
point), 
where $P(z) = \alpha_d z^d + \alpha_{d-1} z^{d-1} + \ldots + \alpha_0$ is a 
polynomial with $\alpha_d \neq 0$. As above, let 
$\cK \coloneq \widehat{\C} \backslash E$ and let
\begin{equation*}
\Phi : \cK \to \cL = \{ w \in \widehat{\C} : \abs{U(w)} > \mu \}
\end{equation*}
be the lemniscatic map of $\cK$. Then the Green's function with pole at 
infinity for $\cK$ 
is given by~\eqref{eqn:green_cK}, and can also be expressed as
\begin{equation*}
g_\cK(z)
= \log \abs{U(\Phi(z))} - \log(\mu)
= \frac{1}{d} g_{\widehat{\C} \backslash \Omega} (P(z))
= \frac{1}{d} \log \abs{ \widetilde{\Phi}(P(z)) };
\end{equation*}
see the proof of Theorem~5.2.5 in~\cite{Ransford1995}.
This shows that $\Phi$ and $\widetilde{\Phi}$ are related by
\begin{equation}
\abs{ U(\Phi(z)) } = \mu \abs{ \widetilde{\Phi}(P(z)) }^{1/d},
\label{eqn:connect_Phi_and_Riemann}
\end{equation}
where 
\begin{equation*}
\mu = c(E) = \left( \frac{c(\Omega)}{\abs{\alpha_d}} \right)^{1/d}
= \left( \frac{1}{ \abs{\alpha_d}\widetilde{\Phi}'(\infty) } \right)^{1/d};
\end{equation*}
see~\cite[Theorem~5.2.5]{Ransford1995}.
If $\widetilde{\Phi}$ and $P$ are known, the 
equality~\eqref{eqn:connect_Phi_and_Riemann} yields a formula for (the modulus 
of) $U \circ \Phi$.  However, this does not lead to \emph{separate} 
expressions for $U$ and $\Phi$.  In other words, we can in general neither 
obtain the lemniscatic domain nor the lemniscatic map directly 
via~\eqref{eqn:connect_Phi_and_Riemann} from the knowledge of 
$\widetilde{\Phi}$ and $P$.

For certain sets $\Omega$ and polynomials $P$, we obtain by a direct 
construction explicit formulas for $U$ and $\Phi$ in terms of the Riemann 
map $\widetilde{\Phi}$.

\begin{theorem} \label{thm:pre-images}
Let $\Omega = \Omega^* \subseteq \C$ be compact and simply connected (not a 
single point) with exterior Riemann map $\widetilde{\Phi}$ as 
in~\eqref{eqn:Riemann_map}.
Let $P(z) = \alpha z^n + \alpha_0$ with $n \geq 2$,
$\alpha_0 \in \R$ to the left of $\Omega$, and $\alpha > 0$.

Then $E \coloneq P^{-1}(\Omega)$ is the disjoint union of $n$ simply 
connected compact sets, and
\begin{align}
&\Phi : \widehat{\C} \backslash E \to \cL = \{ w \in \widehat{\C} : \abs{U(w)} 
> \mu \}, \\
&\Phi(z) = z \Big( \frac{\mu^n}{z^n} [ (\widetilde{\Phi} \circ P)(z) - 
(\widetilde{\Phi} \circ P)(0) ] \Big)^{\frac{1}{n}}, \label{eqn:Phi_pre-images}
\end{align}
is the lemniscatic map of $E$, where we take the principal branch of the $n$th 
root, and where
\begin{equation}
\mu \coloneq \Big( \frac{1}{\alpha \widetilde{\Phi}'(\infty)} 
\Big)^{\frac{1}{n}} > 0, \quad \text{and} \quad
U(w) \coloneq ( w^n + \mu^n (\widetilde{\Phi} \circ P)(0) )^{\frac{1}{n}}.
\label{eqn:mu_U_polypreimage}
\end{equation}
\end{theorem}

\begin{proof}
We construct the lemniscatic map $\Phi$ first in the sector
\begin{equation*}
S = \Big\{ z \in \C \backslash \{ 0 \} : - \frac{\pi}{n} < \arg(z) < 
\frac{\pi}{n} \Big\},
\end{equation*}
and then extend it by the Schwarz reflection principle.

Since $z \in E$ if and only if $z^n \in \frac{1}{\alpha} (\Omega-\alpha_0)$, 
the set $E$ has a single component $E_1$ in the sector $S$, obtained by taking 
the principal branch of the $n$th root.  Note that $E_1 = E_1^* \subseteq S$ is 
again a simply connected compact set.  Then $E = \cup_{j=1}^n e^{i 2 \pi j/n} 
E_1$ is the disjoint union of $n$ simply connected compact sets.

Starting in $S \backslash E_1$, we construct the lemniscatic map as a 
composition of bijective conformal maps, see Figure~\ref{fig:preimage}:
\begin{enumerate}
\item The function $z_1 = P(z)$ maps $S \backslash E_1$ onto the complement of 
$]-\infty, \alpha_0] \cup \Omega$.

\item Then $z_2 = \widetilde{\Phi}(z_1)$ maps this domain onto 
the complement of $]-\infty, \widetilde{\Phi}(\alpha_0)] \cup \{ z_2 : 
\abs{z_2} \leq 1 \}$.
Note that $\Omega = \Omega^*$ implies that $\widetilde{\Phi}(z) = \overline{
\widetilde{\Phi}(\overline{z})}$, so that $]-\infty, \alpha_0]$ is mapped to 
the real line.
In particular, $\widetilde{\Phi}(\alpha_0) < -1$.

\item The function $z_3 = \mu^n ( z_2 - \widetilde{\Phi}(\alpha_0) )$ maps the 
previous domain onto the complement of $]-\infty, 0] \cup \{ z_3 \in \C : 
\abs{z_3 + \mu^n \widetilde{\Phi}(\alpha_0)} \leq \mu^n \}$.  Here $\mu > 0$ is 
defined by~\eqref{eqn:mu_U_polypreimage}.

\item Finally, $w = z_3^{1/n}$, where we take the principal branch of 
the square root, maps this domain onto $S \cap \{ w : \abs{w^n + 
\mu^n \widetilde{\Phi}(\alpha_0) } > \mu^n \}$.
\end{enumerate}

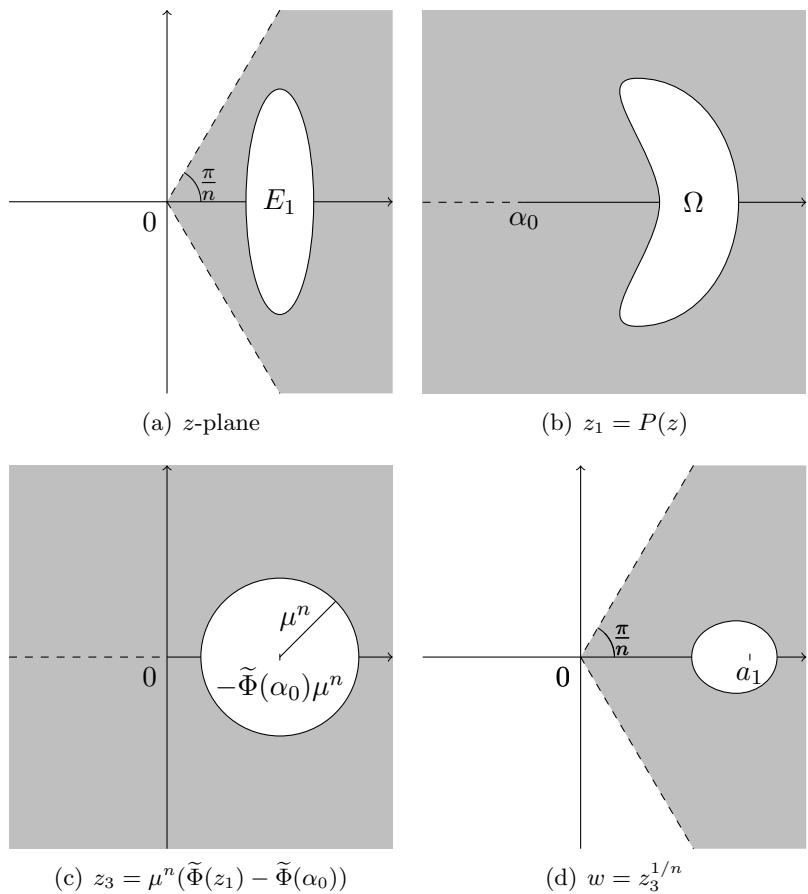
\begin{figure}
\begin{center}
\subfigure[$z$-plane]{
\begin{tikzpicture}[scale = 1.5]
\path [fill=lightgray] (0,0) -- (1,1.7) -- (2,1.7) -- (2,-1.7) -- (1,-1.7) ;
\draw [dashed] (1,1.7) -- (0,0) -- (1,-1.7) ;
\draw [->] (-1.4,0) -- (2,0) ;
\draw [->] ((0,-1.7) -- (0,1.7) ;
\draw (0.3,0) arc [radius = 0.3, start angle=0, end angle = 60] ;
\node [] at (22.5:0.4) {$\frac{\pi}{n}$} ;
\node [below left] at (0,0) {$0$} ;
\draw[fill=white] (1,0) ellipse (0.3cm and 1cm) ;
\node [] at (1,0) {$E_1$} ;
\end{tikzpicture}
}
\subfigure[$z_1 = P(z)$]{
\begin{tikzpicture}[scale = 1.5]
\fill[lightgray] (-1.4,-1.7) rectangle (2, 1.7) ;
\draw [dashed] (-1.4,0) -- (-0.5,0) ;
\draw [->] (-0.5,0) -- (2,0) ;
\node [below] at (-0.5,0) {$\alpha_0$} ;
\draw [fill = white] (1.4,0) to [out = 90, in = 0] (0.5, 1.1)
to [out = 180, in = 90] (0.7, 0) to [out = 270, in = 180] (0.5, -1.1)
to [out = 0, in = 270] (1.4,0) ;
\node [] at (1,0) {$\Omega$} ;
\end{tikzpicture}
}

\subfigure[$z_3 = \mu^n ( \widetilde{\Phi}(z_1) - \widetilde{\Phi}(\alpha_0))$]{
\begin{tikzpicture}[scale = 1.5]
\fill[lightgray] (-1.4,-1.7) rectangle (2, 1.7) ;
\draw [dashed] (-1.4,0) -- (0,0) ;
\draw [->] (0,0) -- (2,0) ;
\draw [->] ((0,-1.7) -- (0,1.7) ;
\node [below left] at (0,0) {$0$} ;
\draw [fill = white] (1,0) circle [radius = 0.7];
\draw (1,-0.03) -- (1,0.03) ;
\node[below] at (1,0) {$-\widetilde{\Phi}(\alpha_0) \mu^n$} ;
\draw (1,0) -- (1.4950,0.4950) ;
\node[above] at (1.14,0.14) {$\mu^n$} ;
\end{tikzpicture}
}
\subfigure[$w = z_3^{1/n}$]{
\begin{tikzpicture}[scale = 1.5]
\path [fill=lightgray] (0,0) -- (1,1.7) -- (2,1.7) -- (2,-1.7) -- (1,-1.7) ;
\draw [dashed] (1,1.7) -- (0,0) -- (1,-1.7) ;
\draw [->] (-1.4,0) -- (2,0) ;
\draw [->] ((0,-1.7) -- (0,1.7) ;
\draw (0.3,0) arc [radius = 0.3, start angle=0, end angle = 60] ;
\node [] at (22.5:0.4) {$\frac{\pi}{n}$} ;
\node [below left] at (0,0) {$0$} ;
\draw [fill = white] (1.7409,0) to [out = 90, in = 0] (1.38, 0.3215)
to [out = 180, in = 90] (0.9843, 0) to [out = 270, in = 180] (1.38, -0.3215)
to [out = 0, in = 270] (1.7409,0);
\draw (1.5,-0.03) -- (1.5,0.03) ;
\node[below] at (1.5,0) {$a_1$} ;
\draw (0.3,0) arc [radius = 0.3, start angle=0, end angle = 45] ;
\node [] at (22.5:0.4) {$\frac{\pi}{n}$} ;
\node [below left] at (0,0) {$0$} ;
\end{tikzpicture}
}
\end{center}
\caption{Construction of the lemniscatic map in the proof of 
Theorem~\ref{thm:pre-images}.}
\label{fig:preimage}
\end{figure}

Since each map is bijective and conformal, their composition $\Phi$ 
is a bijective conformal map from $S \backslash E_1$ to $S \cap \{ 
w \in \widehat{\C} : \abs{U(w)} > \mu \}$.  A short computation shows that
\begin{equation*}
\Phi(z) = ( \mu^n [ (\widetilde{\Phi} \circ P)(z) - (\widetilde{\Phi} 
\circ P)(0) ] )^{\frac{1}{n}} 
= z \Big( \frac{\mu^n}{z^n} [ (\widetilde{\Phi} \circ P)(z) - 
(\widetilde{\Phi} \circ P)(0) ] \Big)^{\frac{1}{n}}
\end{equation*}
for $z \in S \backslash E_1$.
Since $\Phi$ maps the half-lines $\arg(z) = \pm \frac{\pi}{n}$ onto themselves, 
$\Phi$ can be extended by Schwarz' reflection principle to a bijective and 
conformal map from $(\widehat{\C} \backslash E) \backslash \{ 0, \infty \}$ 
to $\cL \backslash \{ 0, \infty \}$.  Note that $\Phi$ is given 
by~\eqref{eqn:Phi_pre-images} for every $z \in (\widehat{\C} \backslash E) 
\backslash \{ 0, \infty \}$, since the right hand side 
of~\eqref{eqn:Phi_pre-images} is analytic there.  This follows since the 
expression under the $n$th root lies in $\C \backslash ]-\infty, 0]$ for every 
$z \in (\widehat{\C} \backslash E) \backslash \{ 0, \infty \}$.

It remains to show that $\Phi$ is defined and conformal in $0$ and $\infty$, 
and that it satisfies the normalization in~\eqref{eqn:normalization_Phi}.
We begin with the point $z = 0$.  Near $\alpha_0$ the Riemann mapping 
$\widetilde{\Phi}$ has the form
\begin{equation*}
\widetilde{\Phi}(z) = \widetilde{\Phi}(\alpha_0) + \widetilde{\Phi}'(\alpha_0) 
(z-\alpha_0) + \cO( (z-\alpha_0)^2 ).
\end{equation*}
Then near $0$ we have 
\begin{equation*}
\Phi(z)
= z \Big( \frac{\mu^n}{z^n} [ \widetilde{\Phi}'(\alpha_0) \alpha z^n + 
\cO(z^{2n}) ] \Big)^{\frac{1}{n}}
= z \Big( \mu^n \widetilde{\Phi}'(\alpha_0) \alpha + \cO(z^n) 
\Big)^{\frac{1}{n}},
\end{equation*}
so that $\Phi(0) = 0$, and $\Phi'(0) = ( \mu^n \widetilde{\Phi}'(\alpha_0) 
\alpha )^{\frac{1}{n}} \neq 0$, showing that $\Phi$ is defined and conformal at 
$0$.

Near $z = \infty$, we have $\widetilde{\Phi}(z) = \widetilde{\Phi}'(\infty) z + 
\cO(1)$, so that, together with~\eqref{eqn:mu_U_polypreimage},
\begin{equation*}
\Phi(z) = z \Big( \frac{\mu^n}{z^n} [ \widetilde{\Phi}'(\infty) \alpha z^n + 
\cO(1) ] \Big)^{\frac{1}{n}}
= z \Big( 1 + \cO \Big( \frac{1}{z^n} \Big) \Big)^{\frac{1}{n}}
= z + \cO \Big( \frac{1}{z^{n-1}} \Big).
\end{equation*}
Thus $\Phi$ satisfies~\eqref{eqn:normalization_Phi} and is a bijective 
conformal map from $\widehat{\C} \backslash E$ to $\cL$, as claimed.
\eop
\end{proof}

The assumption $\alpha > 0$ in Theorem~\ref{thm:pre-images} has been made
for simplicity only. In the notation of the theorem, if 
$P_\theta(z) = \alpha e^{i \theta} z^n + \alpha_0$ with $\theta \in \R$, then 
$P_\theta(z) = P(e^{i \theta/n} z)$. Hence 
$P_\theta^{-1}(\Omega) = e^{- i \theta/n} E = \tau(E)$, with $\tau(z) = 
e^{-i \theta/n} z$.  Then the lemniscatic map of $\widehat{\C} \backslash 
P_\theta^{-1}(\Omega)$ is $\tau \circ \Phi \circ \tau^{-1}$; see 
Lemma~\ref{lem:linear_trafo}.

Also note that if $\Omega$ is symmetric with respect to the line through the 
origin and some point $e^{i \theta}$, then, taking $\alpha_0$ on that line to 
the left of $\Omega$, the assertion of Theorem~\ref{thm:pre-images} remains 
unchanged.

\begin{example}{\rm
As an example we consider the compact set $\Omega$ in 
Figure~\ref{fig:bw_preimage_Omega}, which is of the form introduced 
in~\cite[Theorem~3.1]{KochLiesen2000}.  It is defined with the parameters 
$\lambda = -1$, $\phi = \frac{\pi}{2}$ and $R = 1.1$ through the inverse of 
its Riemann map
\begin{equation*}
\widetilde{\Phi}^{-1}(w) = \frac{ (w-\lambda N)(w-\lambda M)}{(N-M) w + \lambda 
(MN-1)},
\end{equation*}
where
\begin{equation*}
Q = \tan(\phi/4) + \frac{1}{\cos(\phi/4)}, \quad
N = \frac{1}{2} \left( \frac{Q}{R} + \frac{R}{Q} \right), \quad
M = \frac{ R^2 - 1}{2 R \tan(\phi/4)}.
\end{equation*}
Then $\Omega$ is the compact set bounded by $\widetilde{\Phi}^{-1} ( \{ w \in 
\C : \abs{w} = 1 \} )$, and thus has, in particular, an analytic boundary. 
Since $\Omega=\Omega^*$ and $\alpha_0 = 0$ lies to the left of $\Omega$, we can 
apply Theorem~\ref{thm:pre-images} with $P(z) = z^3$.  
Figure~\ref{fig:bw_preimage_E} shows the set $E = P^{-1}(\Omega)$, and 
Figure~\ref{fig:bw_preimage_L} shows the corresponding lemniscatic domain.
}\end{example}

\begin{figure}
\begin{center}
\subfigure[$E = P^{-1}(\Omega)$ with $P(z) = z^3$]{
\label{fig:bw_preimage_E}
\includegraphics[width=0.48\textwidth]{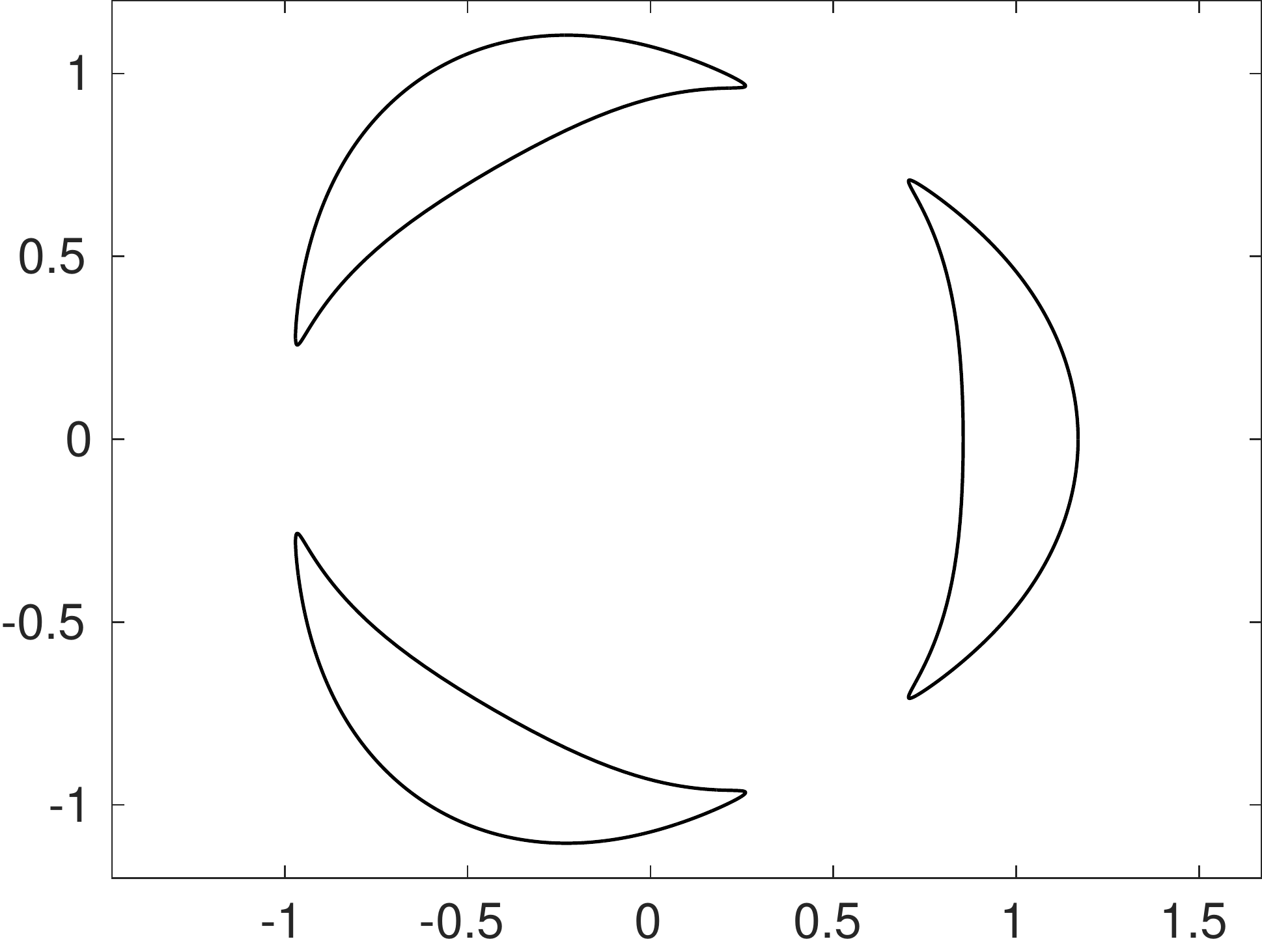}}
\subfigure[Compact set $\Omega = \Omega^*$.]{
\label{fig:bw_preimage_Omega}
\includegraphics[width=0.48\textwidth]{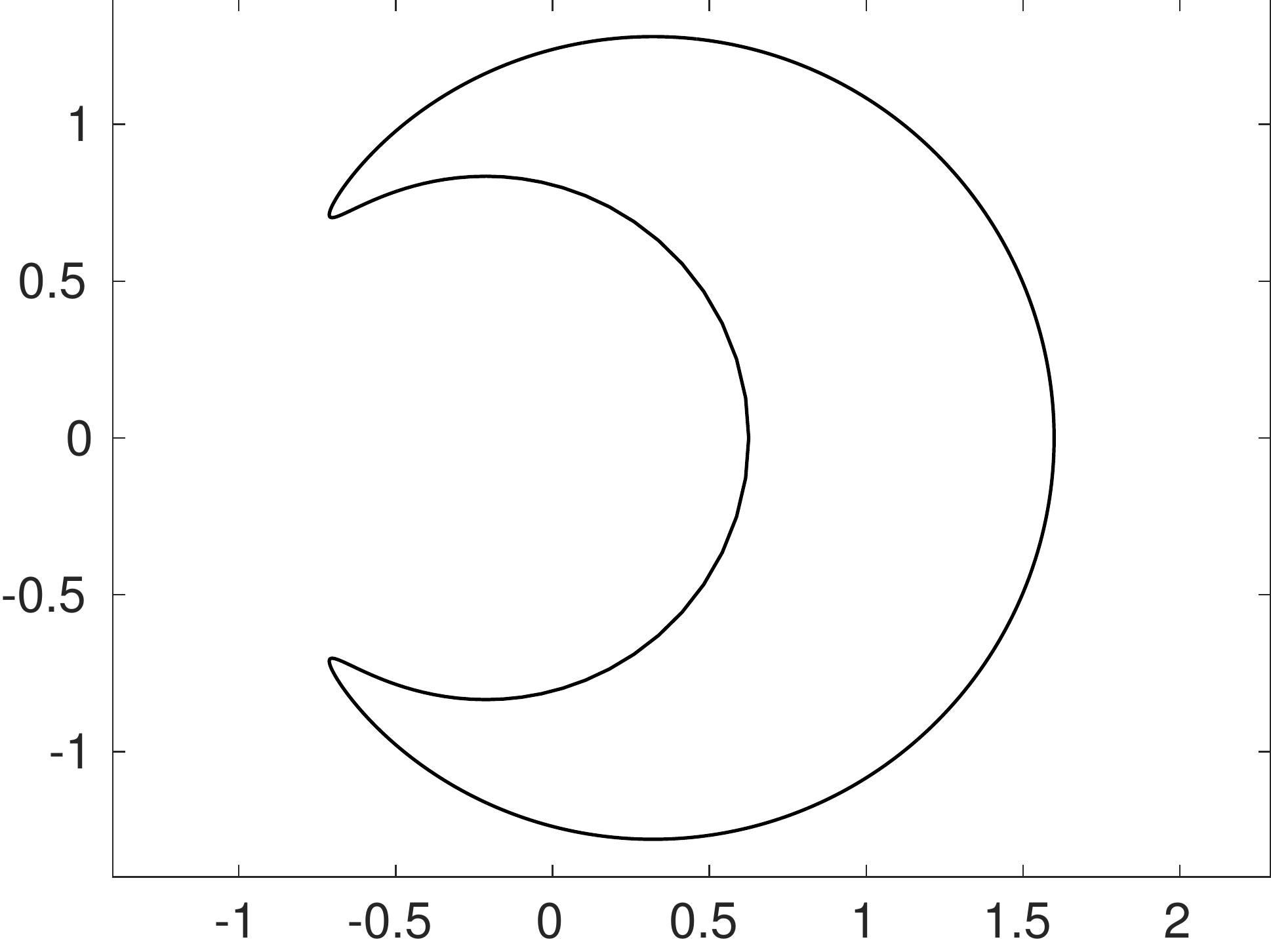}}
\subfigure[Lemniscatic domain $\cL$]{
\label{fig:bw_preimage_L}
\includegraphics[width=0.48\textwidth]{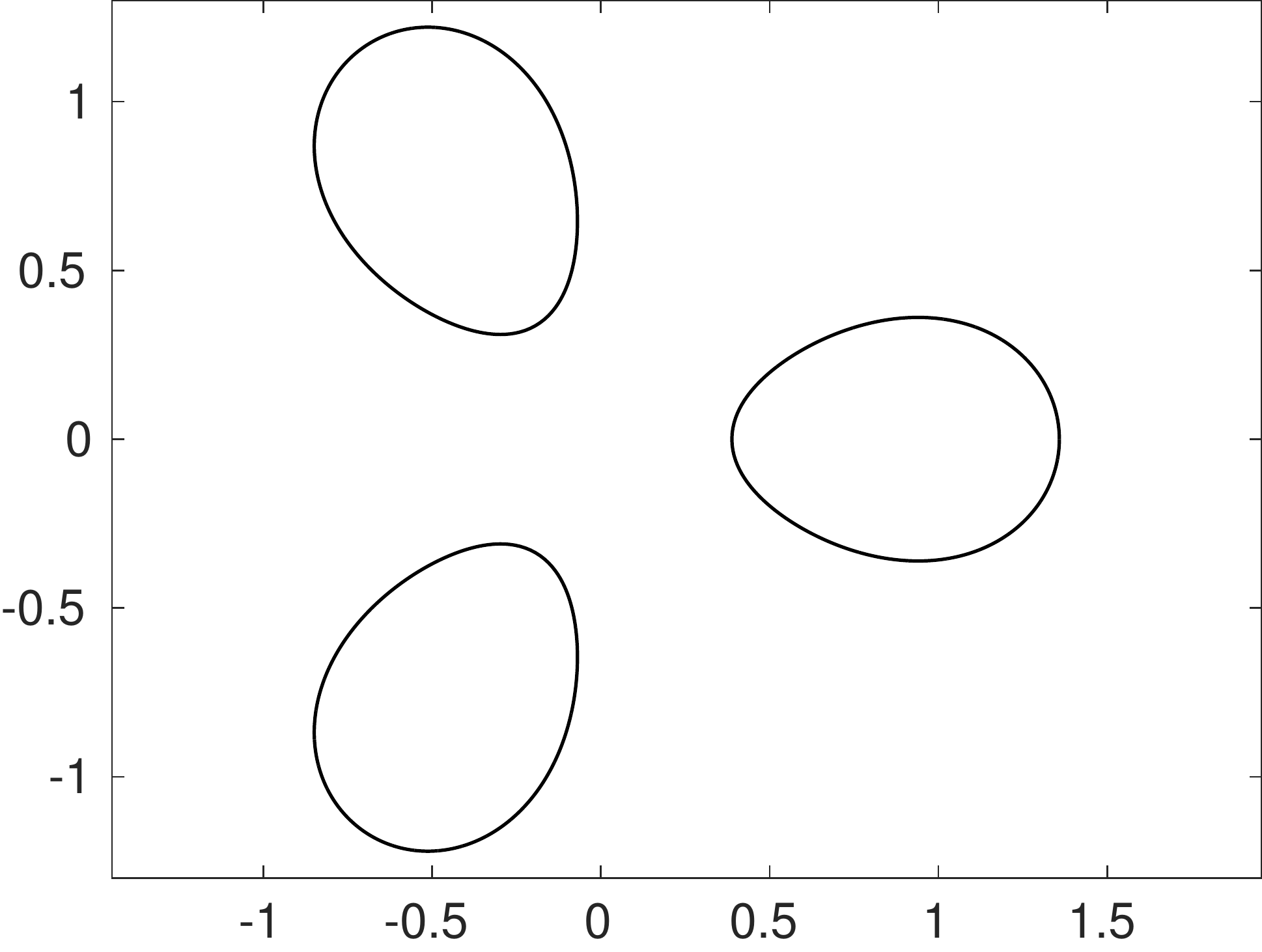}}
\end{center}
\caption{Illustration of Theorem~\ref{thm:pre-images} with a set
from~\cite[Theorem~3.1]{KochLiesen2000}.}
\end{figure}

Using Theorem~\ref{thm:pre-images} we now derive the lemniscatic
conformal map for a radial slit domain.

\begin{corollary} \label{cor:radial_slits}
Let $E = \cup_{j=1}^n e^{i 2 \pi j / n} [C, D]$ with $0 < C < D$.  Then
\begin{equation}
\Phi(z) = z \Big( \frac{1}{2} + \frac{\sqrt{D}^n \sqrt{C}^n}{2} \frac{1}{z^n}
\pm \frac{1}{2 z^n} \sqrt{(z^n-C^n)(z^n-D^n)} \Big)^{\frac{1}{n}}
\label{eqn:Phi_radial_slits}
\end{equation}
is the lemniscatic map of $E$ with corresponding lemniscatic domain
\begin{equation}
\cL = \left\{ w \in \widehat{\C} : \abs{U(w)} = 
\left\lvert w^n - \frac{ (\sqrt{D}^n+\sqrt{C}^n)^2 }{4} 
\right\rvert^{\frac{1}{n}} > \mu = \Big( \frac{D^n-C^n}{4} \Big)^{\frac{1}{n}} 
\right\}.
\label{eqn:cL_radial_slits}
\end{equation}
The inverse of $\Phi$ is given by
\begin{equation}
\Phi^{-1}(w) = w \left( 1 + \frac{ (\sqrt{D}^n-\sqrt{C}^n)^2 }{4} \frac{1}{w^n 
- \frac{ (\sqrt{D}^n+\sqrt{C}^n)^2 }{4}} \right)^{\frac{1}{n}},
\label{eqn:psi_radial_slits}
\end{equation}
where we take the principal branch of the $n$th root.
\end{corollary}

\begin{proof}
With $P(z) = z^n$ and $\Omega = [C^n, D^n]$ we have $E = P^{-1}(\Omega)$, and
Theorem~\ref{thm:pre-images} applies.
We need the conformal map
\begin{equation*}
\widetilde{\Phi} : \widehat{\C} \backslash [C^n,D^n] \to \{ w \in \widehat{\C} 
: \abs{w} > 1 \}, \quad \widetilde{\Phi}(\infty) = \infty, \quad 
\widetilde{\Phi}'(\infty) > 0.
\end{equation*}
Clearly, its inverse is given by
\begin{equation*}
\widetilde{\Phi}^{-1}(w) = \frac{D^n-C^n}{4} \Big( w+\frac{1}{w} \Big) + 
\frac{D^n+C^n}{2},
\end{equation*}
so that
\begin{equation*}
\widetilde{\Phi}(z) = \frac{2}{D^n-C^n} \Big( z - \frac{D^n+C^n}{2} \pm 
\sqrt{(z-C^n)(z-D^n)} \Big),
\end{equation*}
where the branch of the square root is chosen such that 
$\abs{\widetilde{\Phi}(z)} > 1$.
In particular, $\widetilde{\Phi}'(\infty) = \frac{4}{D^n-C^n}$.
Using this in Theorem~\ref{thm:pre-images} 
yields~\eqref{eqn:Phi_radial_slits} and~\eqref{eqn:cL_radial_slits}.
By reversing the construction in the proof of 
Theorem~\ref{thm:pre-images}, we see that
\begin{equation*}
\Phi^{-1}(w) = \Big( \widetilde{\Phi}^{-1} \Big( \frac{w^n}{\mu^n} + 
\widetilde{\Phi}(0) \Big) \Big)^{\frac{1}{n}},
\end{equation*}
which, after a short computation, yields~\eqref{eqn:psi_radial_slits}.
\eop
\end{proof}

Corollary~\ref{cor:radial_slits} shows, in particular, that 
$\left( \frac{D^n-C^n}{4} \right)^{1/n}$ is the logarithmic 
capacity of $E = \cup_{j=1}^n e^{i 2 \pi j/n} [C, D]$; see~\cite{Hasson2003}.

Let us have a closer look at Corollary~\ref{cor:radial_slits} in the case $n = 
2$, i.e.,
\begin{equation*}
E = [-D,-C] \cup [C,D].
\end{equation*}
First note that in this case Corollary~\ref{cor:radial_slits} gives a new proof 
for the well-known fact that the logarithmic capacity of $E$ is $c(E) = 
\tfrac{\sqrt{D^2-C^2}}{2}$; see, e.g.,~\cite[Corollary~5.2.6]{Ransford1995} 
and~\cite{Hasson2003}.
Figure~\ref{fig:psi2ints} shows phase portraits of $\Phi$ and
$\psi \coloneq \Phi^{-1}$ 
for the values $C = 0.1$ and $D = 1$; see~\cite{Wegert2012,WegertSemmler2011} 
for details on phase portraits.  The black lines in the left figure are the two 
intervals forming $E$, and the black curves in the right figure are the 
boundary of $\cL$.  At the black and white dots the functions have the values 
$0$ and $\infty$, respectively. The zeros of $\psi$ are $0$ and $\pm 
\sqrt{DC}$.
The function $\psi : \cL \to \cK$ can be continued analytically (but not 
conformally) to a full neighbourhood of the lemniscate $\{ w : \abs{w^2 - 
\frac{(D+C)^2}{4}} = \frac{D^2-C^2}{4} \}$. The zeros of $\psi'$ are denoted by 
black crosses. Note the discontinuity of the phase of $\psi$ between the zeros and 
singularities interior to the lemniscate. This suggests that $\psi$ will be 
analytic and single-valued in $\{ w \in \widehat{\C} : \abs{U(w)} > 
\tfrac{D-C}{2} \}$.

\begin{figure}
\subfigure[Phase portrait of $\Phi$.]{
\includegraphics[width=0.48\textwidth]{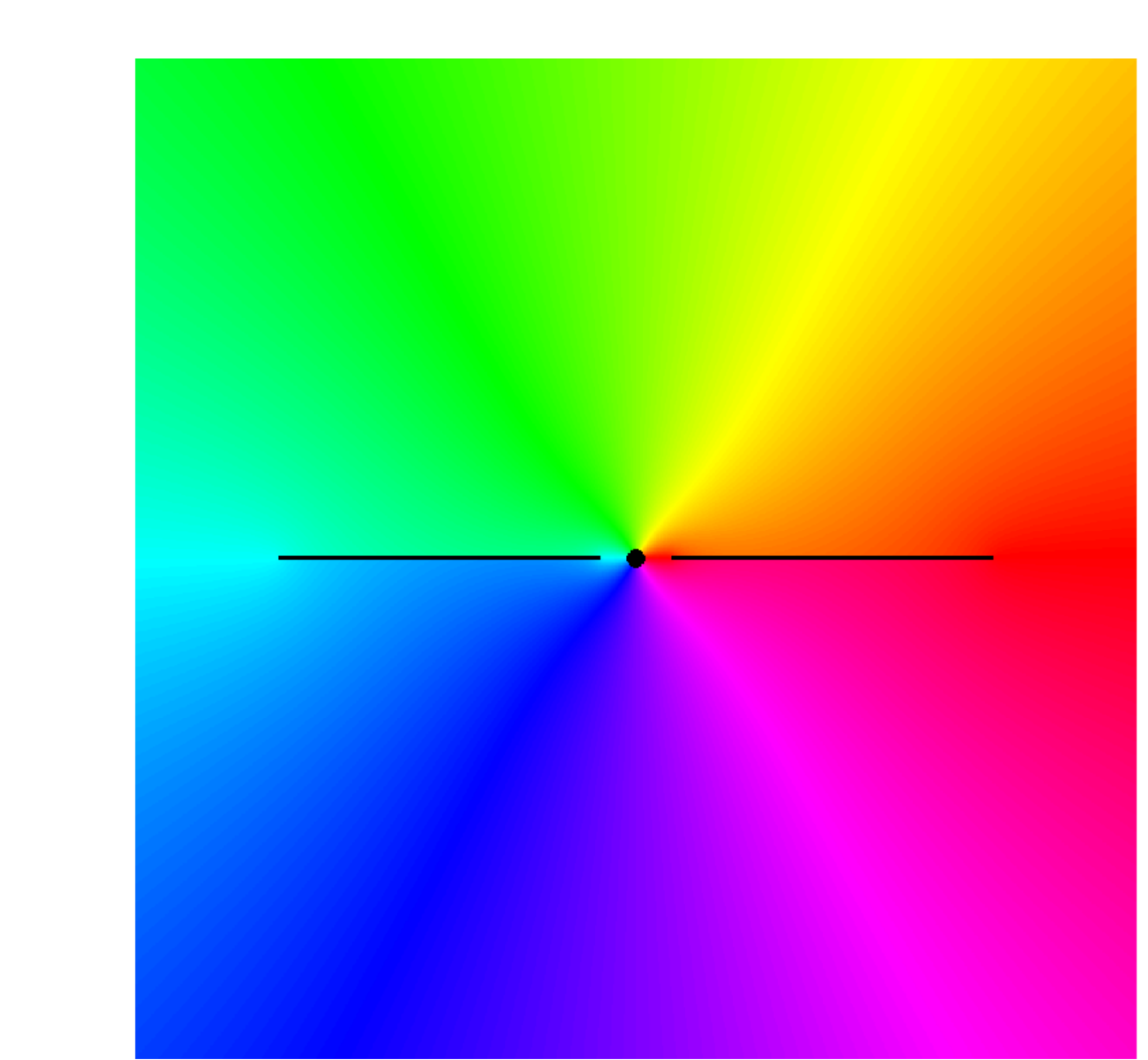}}
\subfigure[Phase portrait of $\Phi^{-1}$.]{
\includegraphics[width=0.48\textwidth]{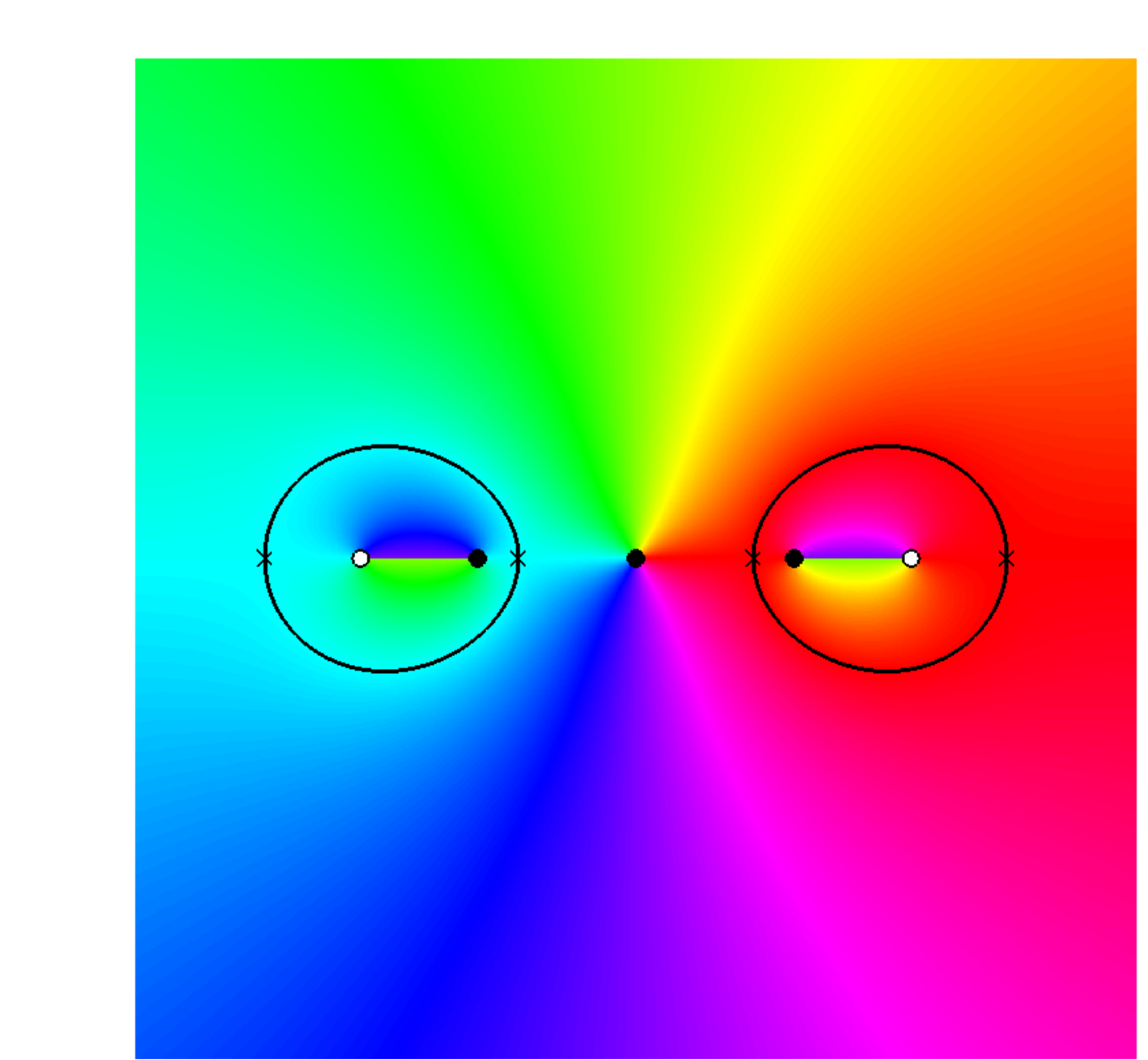}}
\caption{Phase portraits of $\Phi$ and $\Phi^{-1}$ from 
Corollary~\ref{cor:radial_slits} for $n = 2$ and $C = 0.1$ and $D = 1$.}
\label{fig:psi2ints}
\end{figure}


\section{Lemniscatic map for two equal disks}
\label{sect:map_disks}

In this section we analytically construct the lemniscatic map of a set $E$ that 
is the union of two disjoint equal disks.
Let us denote by $D_r(z_0) = \{ z \in \C : \abs{z-z_0} \leq r \}$ the closed 
disk with radius $r > 0$ and center $z_0 \in \C$.
By Lemma~\ref{lem:linear_trafo} we can assume without loss of generality that 
$E = D_r(z_0) \cup D_r(-z_0)$ with real $z_0$ and $0 < r < z_0$.
Let $P(z) = \alpha z^2 + \alpha_0$ with $\alpha > 0$, then
\begin{equation*}
\Omega = \{ \alpha (z_0 + \rho e^{it})^2 + \alpha_0 : 0 \leq \rho \leq r,
0 \leq t \leq 2 \pi \}
\end{equation*}
is a simply connected compact set with $E = P^{-1}(\Omega)$,
so that in principle we could apply Theorem~\ref{thm:pre-images}. 
However, the Riemann map for the set $\Omega$ seems not to be readily 
available.  Therefore, we directly construct the lemniscatic map as a 
composition of certain conformal maps.  The main ingredients are the map from 
the exterior of two disks onto the exterior of two intervals and from there 
onto a lemniscatic domain (according to Corollary~\ref{cor:radial_slits}).

We need the following conformal map from~\cite[pp.~293--295]{Nehari1952}.

\begin{lemma} \label{lem:annulus_to_slit_plane}
Let $0 < \rho < 1$ and define
\begin{equation}
L = L(\rho) \coloneq 2 \rho \prod_{n=1}^\infty \Big( \frac{1 + \rho^{8n}}{1 +
\rho^{8n-4}} \Big)^2, \label{eqn:formula_L}
\end{equation}
and the complete elliptic integral of the first kind
\begin{equation}
K = K(k) \coloneq \int_0^1 \frac{dt}{ \sqrt{(1 - t^2) (1 - k^2 t^2)} } \quad 
\text{with } k \coloneq L^2.
\label{eqn:K}
\end{equation}
Then the function
\begin{equation}
w = f(z) = L \sn \Big(\frac{2K}{\pi} i \log \Big(\frac{z}{\rho}\Big) + K; k\Big)
\label{eqn:annu_to_slit_disc}
\end{equation}
is a bijective and conformal map from the annulus $\rho < \abs{z} < \rho^{-1}$ 
onto the $w$-plane with the slits $-\infty < w \leq -\tfrac{1}{L}$, $-L \leq w 
\leq L$ and $\tfrac{1}{L} \leq w < \infty$.
Further, we have $f(-1) = -1$ and $f'(-1) = (1-L^2) \tfrac{2K}{\pi} > 0$, and
$f(z^{-1}) = (f(z))^{-1}$.
\end{lemma}

\begin{proof}
See~\cite[pp.~293--295]{Nehari1952} for the existence and mapping properties of
$f$.  Note that $f$ is independent of the choice of the branch of the 
logarithm.  By construction $f$ is also symmetric with respect to the real axis 
and to the unit circle, i.e., $f(z) = \overline{ f(\overline{z}) }$, and
$f(z) = 1 / \overline{ f( 1/\overline{z} ) }$.
This implies $f(1/z) = 1 / \overline{ f(\overline{z}) } = 1/f(z)$.

It remains to compute $f(-1)$ and $f'(-1)$.  Recall the identity
\begin{equation*}
\sn'(z; k) = \cn(z; k) \dn(z; k),
\end{equation*}
where $\cn(z; k) = \sqrt{1 - \sn(z; k)^2}$, 
$\cn(0) = 1$, 
and $\dn(z) = \sqrt{1 - k^2 \sn(z; k)^2}$, $\dn(0) = 1$.  We compute
\begin{equation*}
f'(z) = L \cn(\zeta(z); k) \dn(\zeta(z); k) \frac{2K}{\pi} i \frac{1}{z}, \quad 
\text{where} \quad
\zeta(z) = \frac{2 K}{\pi} i \log \Big( - i \frac{z}{\rho} \Big).
\end{equation*}
For $z = -1$ we have $\zeta(-1) = -K - i \frac{2K}{\pi} \log(\rho) \eqcolon -K 
+ i \frac{K'}{2}$; see~\cite[p.~294]{Nehari1952}.  With the special values
\begin{equation*}
\begin{split}
\sn(K + i \tfrac{K'}{2}; k) &= \frac{1}{\sqrt{k}}, \quad
\cn(K + i \tfrac{K'}{2}; k) = -i \sqrt{\frac{1}{k}-1}, \\
\dn(K + i \tfrac{K'}{2}; k) &= \sqrt{1-k},
\end{split}
\end{equation*}
see \cite[p.~381]{MagnusOberhettingerSoni1966} or \cite[p.~145]{Tricomi1948},
and the identities
\begin{equation*}
\begin{split}
\sn(z+2K; k) &= - \sn(z; k), \quad
\cn(z+2K; k) = - \cn(z; k), \\
\dn(z+2K; k) &= \dn(z; k),
\end{split}
\end{equation*}
see \cite[p.~500]{WhittakerWatson1962}, we obtain
\begin{equation*}
\begin{split}
f(-1) &= L \sn \Big( - K + i \frac{K'}{2}; k \Big) = - L \frac{1}{\sqrt{k}}
= -1, \\
f'(-1) &= L i \sqrt{\frac{1-k}{k}} \sqrt{1-k} \frac{2K}{\pi} i (-1)
= L \frac{1-k}{\sqrt{k}} \frac{2K}{\pi} = (1-L^2) \frac{2K}{\pi} > 0.
\end{split}
\end{equation*}
In the last equalities we used $k = L^2$.
\eop
\end{proof}

We now construct the lemniscatic map of the exterior of two disjoint equal 
disks.

\begin{theorem} \label{thm:ext_map_2_disks}
Let $r, z_0 \in \R$ with $0 < r < z_0$, and $E = D_r(z_0) \cup D_r(-z_0)$.
Let $T$ be the M\"{o}bius transformation
\begin{equation*}
T(z) = \frac{\alpha+z}{\alpha-z}, \quad \alpha = \sqrt{z_0^2 - r^2} > 0, 
\end{equation*}
$f$, $K$, $L$ be given as in Lemma~\ref{lem:annulus_to_slit_plane} with
\begin{equation*}
0 < \rho = \frac{\sqrt{z_0+r} - \sqrt{z_0-r}}{\sqrt{z_0+r} + \sqrt{z_0-r}} < 1,
\end{equation*}
and let $\Phi_1$ be the lemniscatic map from~\eqref{eqn:Phi_radial_slits} for
$n = 2$, with
\begin{equation}
C = \frac{2 K \alpha}{\pi} (1-L)^2, \quad
D = \frac{2 K \alpha}{\pi} (1+L)^2.
\label{eqn:CD_twodisks}
\end{equation}

Then
\begin{equation}
\Phi(z) = \Phi_1 ( f'(-1) \cdot (T^{-1} \circ f \circ T)(z) )
\label{eqn:Phi_disks}
\end{equation}
is the lemniscatic map of $E$ with corresponding lemniscatic domain
\begin{equation}
\cL = \left\{ w \in \widehat{\C} :
\Big\lvert w^2 - \Big( \frac{2K\alpha}{\pi} (1+L^2) \Big)^2 
\Big\rvert^\frac{1}{2}
> \sqrt{2 L (1+L^2)} \frac{2 K \alpha}{\pi} \right\},
\label{eqn:cL_two_disks}
\end{equation}
and hence, in particular, $c(E) = \sqrt{2 L (1+L^2)} \frac{2 K \alpha}{\pi}$.
\end{theorem}

\begin{proof}
Our proof is constructive.  First, $\Phi$ is obtained as composition of 
conformal maps which map $\widehat{\C} \backslash E$ to a lemniscatic domain.  
In a second step, we show that $\Phi$ is normalized as 
in~\eqref{eqn:normalization_Phi}, and thus is a lemniscatic map.
The first steps in the construction, namely $T^{-1} \circ f \circ T$, modify 
and generalize a conformal map in~\cite[p.~297]{Nehari1952}, and are 
illustrated in Figure~\ref{fig:map_two_disks}.

\begin{figure}
\begin{center}
\subfigure[exterior of two disks in the $z$-plane]{
\begin{tikzpicture}[scale = 1.25]
\path[use as bounding box] (-2.2,-2.2) rectangle (2.2,2.2) ;
\fill[gray] (-2.2,-1.2) rectangle (0,1.2) ;
\fill[lightgray] (0,-1.2) rectangle (2.2,1.2) ;
\draw [fill=white] (-1,0) circle [radius = 0.5] ;
\draw [fill=white] (1,0) circle [radius = 0.5] ;
\draw (-2.2,0) -- (2.2,0) ;
\draw (0,-1.2) -- (0,1.2) ;
\draw (-1,-0.03) -- (-1,0.03) ;
\node[below] at (-1,0) {$-z_0$} ;
\draw (-1,0) -- (-0.6464, 0.3536) ;
\node[left] at (-0.82, 0.18) {$r$} ;
\draw (1,-0.03) -- (1,0.03) ;
\node[below] at (1,0) {$z_0$} ;
\draw (1,0) -- (1.3536, 0.3536) ;
\node[left] at (1.18, 0.18) {$r$} ;
\end{tikzpicture}
}
\hfill
\subfigure[$z_1 = T(z) = \tfrac{\alpha+z}{\alpha-z}$]{
\begin{tikzpicture}[scale = 1.25]
\draw[fill=lightgray] (0,0) circle [radius = 2] ;
\draw[dashed, fill=gray] (0,0) circle [radius = 1.25] ;
\draw[fill=white] (0,0) circle [radius = 0.5] ;
\draw [->] (-2.2,0) -- (2.2,0) ;
\draw [->] (0,-2.2) -- (0,2.2) ;
\draw (-2,-0.03) -- (-2,0.03) ;
\node[below right] at (-2,0) {$-\tfrac{1}{\rho}$} ;
\draw (-1.25,-0.03) -- (-1.25,0.03) ;
\node[below right] at (-1.25,0) {$-1$} ;
\draw (-0.5,-0.03) -- (-0.5,0.03) ;
\node[below right] at (-0.5,0) {$-\rho$} ;
\draw (0.5,-0.03) -- (0.5,0.03) ;
\node[below left] at (0.5,0) {$\rho$} ;
\draw (1.25,-0.03) -- (1.25,0.03) ;
\node[below left] at (1.25,0) {$1$} ;
\draw (2,-0.03) -- (2,0.03) ;
\node[below left] at (2,0) {$\tfrac{1}{\rho}$} ;
\end{tikzpicture}
}

\subfigure[$z_2 = f(z_1)$]{
\begin{tikzpicture}[scale = 1.25]
\fill[lightgray] (-2.2,-1.2) rectangle (2.2,1.2) ;
\draw[dashed, fill=gray] (0,0) circle [radius = 1] ;
\draw (-2.2,0) -- (2.2,0) ;
\draw [ultra thick] (-2.2,0) -- (-1.5,0) ;
\draw [ultra thick] (-0.5,0) -- (0.5,0) ;
\draw [ultra thick] (1.5,0) -- (2.2,0) ;
\draw (-1.5,-0.03) -- (-1.5,0.03) ;
\node[below] at (-1.5,0) {$-\tfrac{1}{L}$} ;
\draw (-1,-0.03) -- (-1,0.03) ;
\node[below] at (-1,0) {$-1$} ;
\draw (-0.5,-0.03) -- (-0.5,0.03) ;
\node[below] at (-0.5,0) {$-L$} ;
\draw (0.5,-0.03) -- (0.5,0.03) ;
\node[below] at (0.5,0) {$L$} ;
\draw (1,-0.03) -- (1,0.03) ;
\node[below] at (1,0) {$1$} ;
\draw (1.5,-0.03) -- (1.5,0.03) ;
\node[below] at (1.5,0) {$\tfrac{1}{L}$} ;
\end{tikzpicture}
}
\hfill
\subfigure[$z_3 = T^{-1}(z_2) = \alpha \tfrac{z_2-1}{z_2+1}$]{
\begin{tikzpicture}[scale = 1.25]
\fill[gray] (-2.2,-1.2) rectangle (0,1.2) ;
\fill[lightgray] (0,-1.2) rectangle (2.2,1.2) ;
\draw [ultra thick] (-1.5,0) -- (-0.5,0) ;
\draw [ultra thick] (0.5,0) -- (1.5,0) ;
\draw (-2.2,0) -- (2.2,0) ;
\draw (0,-1.2) -- (0,1.2) ;
\draw (-1.5,-0.03) -- (-1.5,0.03) ;
\node[below] at (-1.5,0) {$-\tfrac{\alpha}{b}$} ;
\draw (-0.5,-0.03) -- (-0.5,0.03) ;
\node[below] at (-0.5,0) {$- \alpha b$} ;
\draw (0.5,-0.03) -- (0.5,0.03) ;
\node[below] at (0.5,0) {$\alpha b$} ;
\draw (1.5,-0.03) -- (1.5,0.03) ;
\node[below] at (1.5,0) {$\tfrac{\alpha}{b}$} ;
\end{tikzpicture}
}

\end{center}
\caption{Conformal map from the exterior of two disks to the exterior of two 
intervals; see the proof of Theorem~\ref{thm:ext_map_2_disks}.}
\label{fig:map_two_disks}
\end{figure}
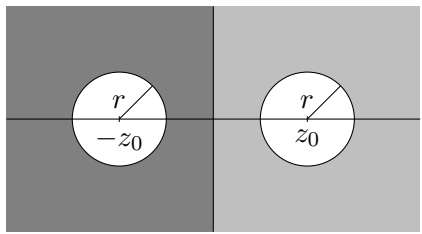
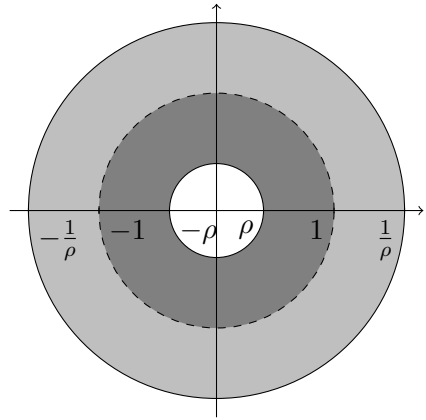
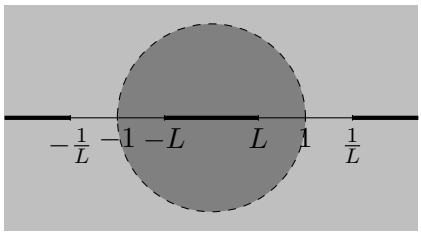
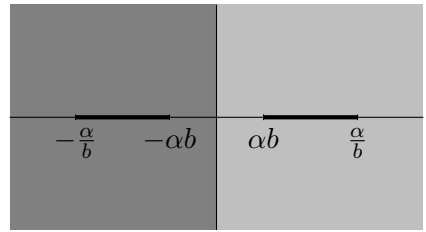

Since $T$ maps the points $-\alpha, 0, \alpha$ to $0, 1, \infty$, respectively, 
$T$ maps $\R$ to $\R$ (with same orientation).
We compute the images of the two disks under $z_1 = T(z)$.
Let
\begin{equation*}
\rho \coloneq T(-z_0+r)
= \frac{ \sqrt{z_0+r} - \sqrt{z_0-r} }{\sqrt{z_0+r} + \sqrt{z_0-r}} \in\; ]0,1[.
\end{equation*}
A short computation shows that $T(-z_0-r) = -\rho$.
Since the circle $\abs{z+z_0} = r$ cuts the real line in a right angle, this 
holds true for its image under $T$, and $T$ maps the circle 
$\abs{z+z_0} = r$ onto the circle $\abs{z_1} = \rho$.
Further, $T(-z) = 1/T(z)$ implies that $T$ maps $\abs{z-z_0} = r$ to 
$\abs{z_1} = \frac{1}{\rho}$.  Hence we see that $T$ maps $\widehat{\C} 
\backslash E$ onto the annulus $\frac{1}{\rho} < \abs{z_1} < \rho$.

This annulus is mapped by $z_2 = f(z_1)$ onto the complex plane with the slits 
$-\infty \leq z_2 \leq -\tfrac{1}{L}$, $-L \leq z_2 \leq L$ and $\tfrac{1}{L} 
\leq z_2 \leq \infty$, where $L = L(\rho)$ is given by~\eqref{eqn:formula_L}; 
see Lemma~\ref{lem:annulus_to_slit_plane}.

For $T^{-1}(z_2) = \alpha \frac{z_2-1}{z_2+1}$ we have $T^{-1}(1/z_2) = - 
T^{-1}(z_2)$.  Then, setting for brevity $b = \tfrac{1-L}{1+L}$, we compute
\begin{equation*}
T^{-1}(L^{-1}) = \alpha b, \;
T^{-1}(-L^{-1}) = \alpha b^{-1}, \;
T^{-1}(-L) = - \alpha b^{-1}, \;
T^{-1}(L) = - \alpha b.
\end{equation*}
This shows that $T^{-1}$ maps the $z_2$-plane with the slits $-\infty \leq z_2 
\leq -L^{-1}$, $-L \leq z_2 \leq L$ and $L^{-1} \leq z_2 \leq \infty$ onto the 
$z_3$-plane with the two slits $[- \alpha b^{-1}, -\alpha b]$ and 
$[\alpha b, \alpha b^{-1}]$.
Multiplying with $f'(-1)$ we obtain the exterior of $[-D, -C] \cup [C, D]$, 
with $C$ and $D$ as in~\eqref{eqn:CD_twodisks}.
The lemniscatic map for this set is $\Phi_1$ from~\eqref{eqn:Phi_radial_slits}
with lemniscatic domain $\cL$ given by~\eqref{eqn:cL_radial_slits}.
A short calculation shows that $\cL$ has the form~\eqref{eqn:cL_two_disks}.

This shows that $\Phi : \widehat{\C} \backslash E \to \cL$ is a bijective and 
conformal map onto a lemniscatic domain, and it remains to 
verify~\eqref{eqn:normalization_Phi}.

We have $\Phi(\infty) = \infty$, since $T(\infty) = -1$ and $f(-1) = -1$, see 
Lemma~\ref{lem:annulus_to_slit_plane}, and since $\Phi_1$ 
satisfies the normalization in~\eqref{eqn:normalization_Phi}.
Next we show that $\Phi'(\infty) = 1$.  Let us begin with the derivative of
$g = T^{-1} \circ f \circ T$ at $z \neq \infty$, which is
\begin{equation*}
g'(z) = (T^{-1})'(f(T(z))) \cdot f'(T(z)) \cdot T'(z).
\end{equation*}
We compute $T'(z) = \frac{2 \alpha}{(\alpha-z)^2}$ and $(T^{-1})'(z) =
\frac{2 \alpha}{(z+1)^2}$, so that
\begin{equation*}
\begin{split}
(T^{-1})'(f(T(z))) \cdot T'(z)
&= (f(T(z)) - f(-1))^{-2} \frac{4 \alpha^2}{(\alpha-z)^2} \\
&= \Big( \frac{f(T(z)) - f(-1)}{T(z)-(-1)} \Big)^{-2}
\frac{4 \alpha^2}{ (T(z)+1)^2 (\alpha-z)^2} \\
&= \Big( \frac{f(T(z)) - f(-1)}{T(z)-(-1)} \Big)^{-2}.
\end{split}
\end{equation*}
We therefore find
\begin{equation*}
g'(\infty) = \lim_{z \to \infty} g'(z) 
= \lim_{z \to \infty} f'(T(z)) \Big( \frac{f(T(z)) - f(-1)}{T(z)-(-1)} 
\Big)^{-2} = \frac{1}{f'(-1)}.
\end{equation*}
This implies $\Phi'(\infty) = \Phi_1'(\infty) f'(-1) g'(\infty) = 1$, so that
$\Phi(z) = z + \cO(1)$ near infinity.
We further show that $\Phi$ is odd, so that the constant term in the Laurent 
series at infinity vanishes, showing~\eqref{eqn:normalization_Phi}.
The function $f$ satisfies $f(1/z) = 1/f(z)$; see 
Lemma~\ref{lem:annulus_to_slit_plane}.  Together with $T(-z) = 
1/T(z)$ and $T^{-1}(1/w) = - T^{-1}(w)$ this gives
\begin{equation*}
g(-z) = T^{-1}( f( T(-z) ) ) = T^{-1}( f(1/T(z)) ) = T^{-1} ( 1/f(T(z)) )
= - g(z).
\end{equation*}
Since also $\Phi_1$ is odd, which can be seen either from 
Lemma~\ref{lem:E_symm} or directly from~\eqref{eqn:Phi_radial_slits}, 
$\Phi$ is an odd function and is normalized as
in~\eqref{eqn:normalization_Phi}.
\eop
\end{proof}

Note that the construction in the proof of Theorem~\ref{thm:ext_map_2_disks} 
can be generalized to doubly connected domains $\cK = \widehat{\C} \backslash 
E$ as follows.
Let $h : \cK \to \{ w \in \C : 1/\rho < \abs{w} < \rho \}$ be a bijective 
conformal map that satisfies $\abs{h(\infty)} = 1$.  In this case we can assume 
(after rotation) that $h(\infty) = -1$.  We then have
\begin{equation*}
h(z) = -1 + a_1 / z + \cO(1/z^2) \quad \text{for } z \text{ near infinity},
\end{equation*}
with $a_1 \neq 0$ since $h$ is conformal.
Let $S(z) = \frac{z-1}{z+1}$.  Then the lemniscatic map of $E$ is given by
\begin{equation*}
\Phi(z) = \Phi_1 \left( - \frac{a_1 f'(-1)}{2} (S \circ f \circ h)(z) \right) + 
\beta
\end{equation*}
with $f$ as in Lemma~\ref{lem:annulus_to_slit_plane}, $\Phi_1$ the lemniscatic 
map of two (possibly rotated) intervals of same length, and $\beta \in \C$ is 
chosen so that the normalization~\eqref{eqn:normalization_Phi} holds.

In Figure~\ref{fig:disks} we plot the sets $E = E(z_0, r)$ for $z_0 = 1$ and $r 
= 0.5$, $0.7$ and $0.9$ (left) and the corresponding lemniscatic domains 
(right).  We evaluated the complete elliptic integral of the first 
kind~\eqref{eqn:K} using the MATLAB function \verb|ellipk|.  The product in the 
formula~\eqref{eqn:formula_L} for $L$ converges very quickly, so that it 
suffices to compute the first few terms in order to obtain the correct value 
up to machine precision.

\begin{figure}
\centerline{
\includegraphics[width=0.5\textwidth]{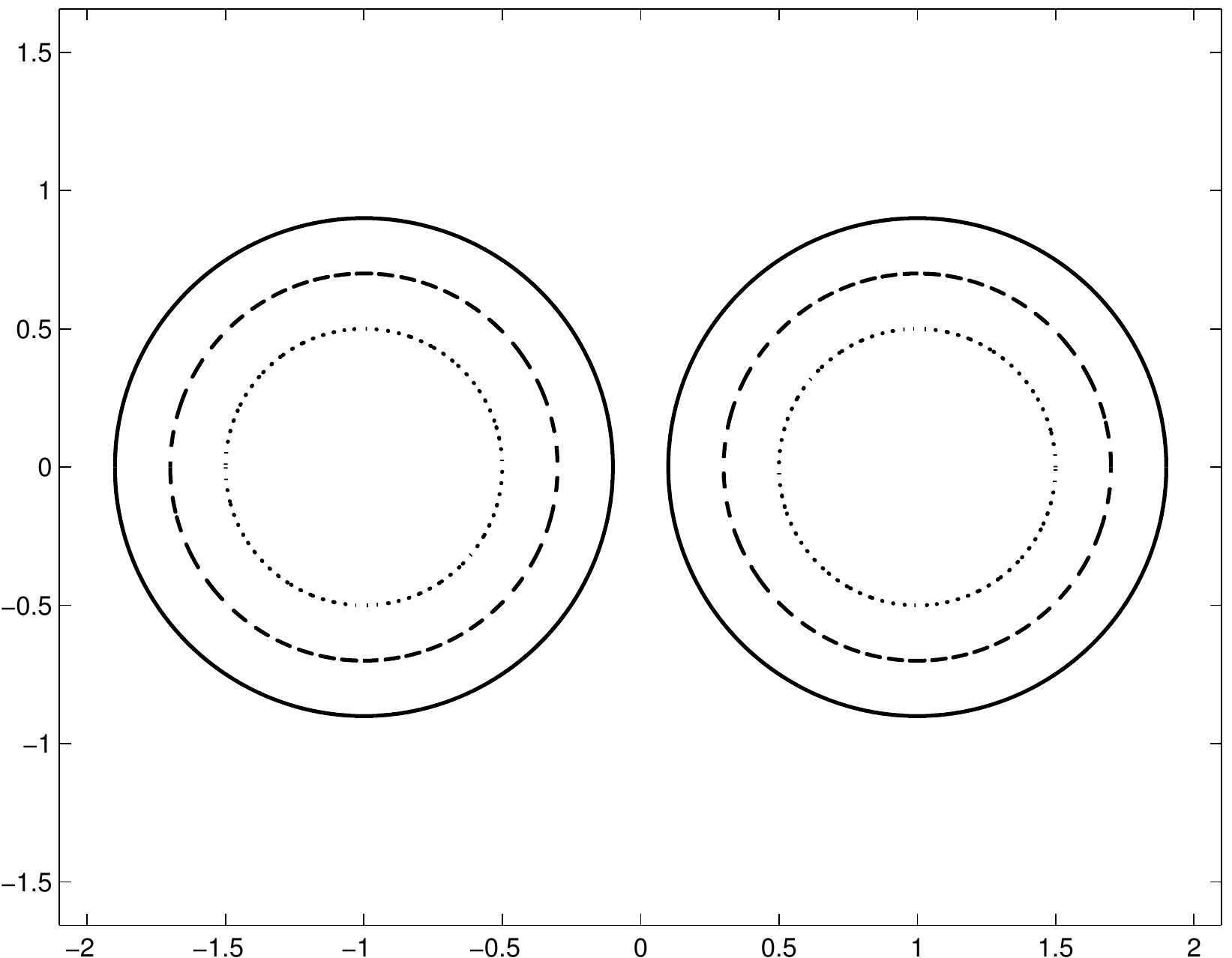}
\includegraphics[width=0.5\textwidth]{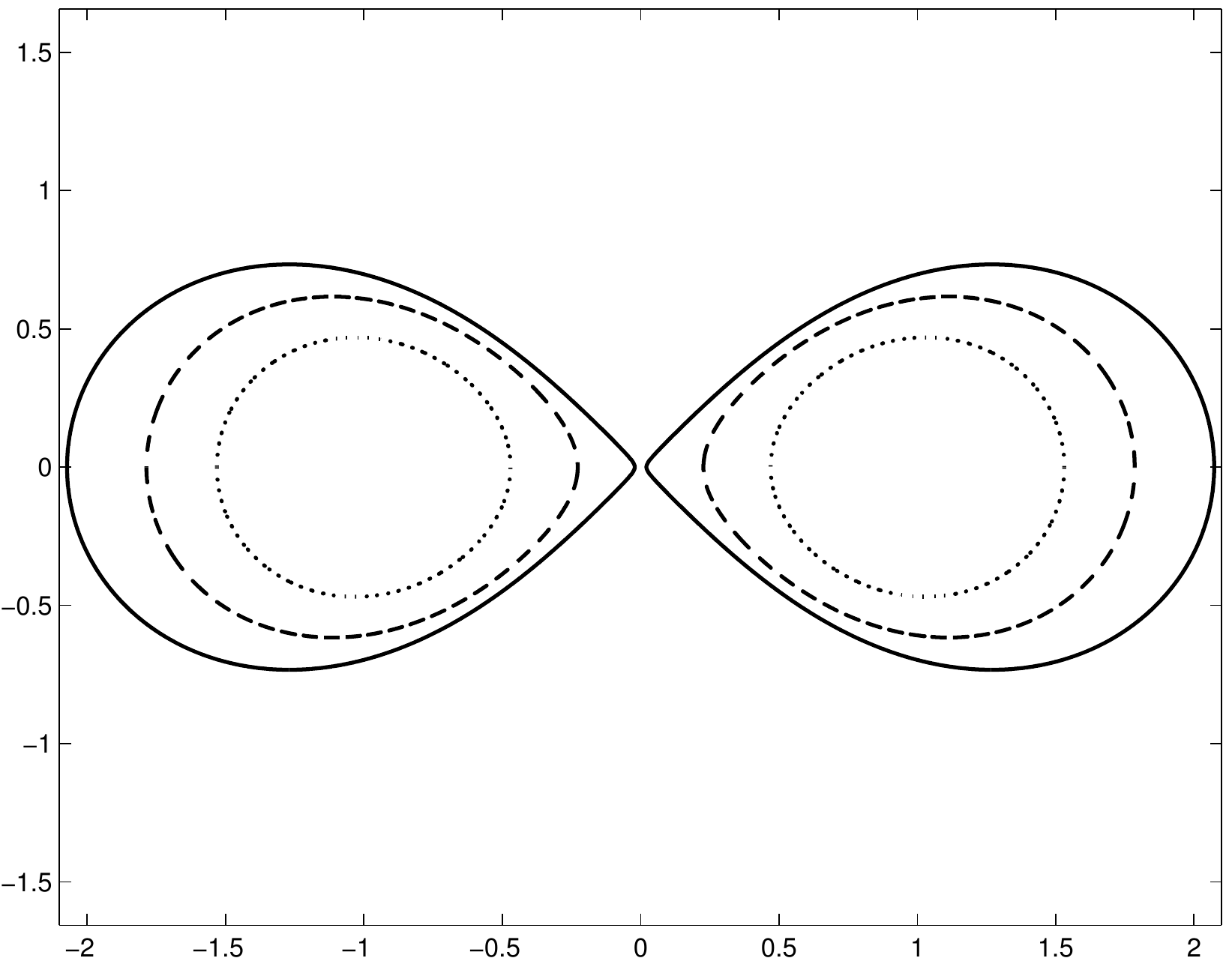}
}
\caption{Illustration of Theorem~\ref{thm:ext_map_2_disks}.}
\label{fig:disks}
\end{figure}


\section{Concluding remarks}
\label{sect:concl}

In this article we investigated properties of lemniscatic maps, i.e., 
conformal maps from multiply connected domains in the extended complex plane
onto lemniscatic domains. We derived a general construction principle 
of lemniscatic maps in terms of the Riemann map for certain polynomial 
pre-images of simply connected sets, and we constructed the first (to our 
knowledge) analytic examples: One for the exterior of $n$ radial slits, 
and one for the exterior of two disks.

Lemniscatic maps allow the construction of the Faber--Walsh polynomials, 
which are a direct generalization of the classical Faber polynomials 
to compact sets consisting of several components. A study of these polynomials 
is given in our paper~\cite{SetLie15}.  Moreover, 
we have addressed the numerical computation of lemniscatic maps 
in~\cite{NasSetLie15}.


\textbf{Acknowledgements.}
We thank Bernd Beckermann for many helpful remarks, and in particular 
for suggesting to generalize the construction of the conformal map for 
the radial slit domain in Corollary~\ref{cor:radial_slits} to 
Theorem~\ref{thm:pre-images}.
We also thank the anonymous referees for helpful comments.

\bibliographystyle{siam}
\bibliography{conformal}

\end{document}